\documentclass[11pt]{article}
\usepackage{amssymb}
\usepackage{amsmath}
\usepackage{color}
\usepackage{hyperref}
\usepackage{csquotes}
\usepackage{IEEEtrantools}

\newenvironment{proof}{\noindent {\bf Proof }}
{\hfill $\bullet$ \vspace{0.25cm}}

\newcommand{\dis}{\displaystyle}

\newtheorem{thm}{Theorem}
\newtheorem{prop}{\indent Proposition}

\newtheorem{rem}{\indent Remark}
\newtheorem{lem}{\indent Lemma}

\newcommand{\mmmintone}[1]{{\dis{\int\kern -.36cm-}}_{\kern-.21cm\substack{#1}}\;\;}
\newcommand{\mmmintwo}[2]{{\dis{\int\kern -.43cm-}}_{\kern-.21cm\substack{#1}}^{\substack{#2}}\;\;}
\newcommand{\submint}{{\scriptstyle{\int\kern -.66em -}}}
\newcommand{\submintone}[1]{{\scriptstyle{\int\kern -.66em-}}_{\scriptscriptstyle{\kern-.21em\substack{#1}}}}
\newcommand{\fracmint}{{\textstyle{\int\kern -.88em -}}}
\newcommand{\fracmintone}[1]{{\textstyle{\int\kern -.88em
-}}_{\scriptscriptstyle{\kern-.21em\substack{#1}}}\;}

\title{The first passage time density of Brownian motion and the heat equation with Dirichlet boundary condition in time dependent domains   }

\author{JM Lee\footnote{  E-mail: ljm9667@gmail.com }\\ \\Seoul, Republic of Korea}

\date{\today}

\begin{document}

\maketitle

\begin{abstract}
In \cite{JMLee}, it is proved that we can have a continuous first-passage-time density function of one dimensional standard Brownian motion when the boundary is H\"{o}lder continuous with exponent greater than $1/2$. For the purpose of extending \cite{JMLee} into multidimensional domains, we show that there exists a continuous first-passage-time density function of standard $d$-dimensional Brownian motion in moving boundaries in $\displaystyle\mathbb{R}^{d}$, $d\geq 2$, under a $C^{3}$-diffeomorphism. Similarly as in \cite{JMLee}, by using a property of local time of standard $d$-dimensional Brownian motion and the heat equation with Dirichlet boundary condition, we find a sufficient condition for the existence of the continuous density function.
\end{abstract}
\vskip 1cm

\section{Introduction}
\label{Fsec.1}
\ First passage time (FPT) problem, which is also called boundary crossing problem, is the one of classical subjects in probability which has also many applications to other fields, for example, finance and biology. There are a bunch of articles studying this problem, but especially we mention one of them, \cite{JMLee}, which is about we can have a continuous first-passage-time density function of one dimensional standard Brownian motion when the boundary is H\"{o}lder continuous with exponent greater than $1/2$. The purpose of this paper is that we extend the result and the general strategy of \cite{JMLee} into the multidimensional domain, precisely, to find a continuous density function of the first hitting time in a time varying domain in $\mathbb{R}^{d}$, $d\geq 2$, by investigating a relation between the first passage time density and the derivative at the boundary of the solution of the heat equation with Dirichlet boundary condition. Thus this article is concerned with standard $d$-dimensional Brownian motion killed on the boundary of a deterministic moving domain which corresponds to the heat equation with Dirichlet boundary condition. In \cite{BCS}, it is studied for reflected Brownian motion whose analytic counterpart is the heat equation with Neumann boundary condition.

\ Unfortunately, in general, it is hard to obtain the explicit form of the density function. For one dimensional Brownian motion, some analytic solutions are introduced in \cite{JMLee}. For two dimensional Brownian motion, in \cite{KZ}, the analytical solution of the Laplace transform of FPT distributions and its inverse is done numerically and other existing literatures including applications to quantitative finance are also summarized.

\ The organization of this article is as follows: In Section \ref{Fsec.2}, we set up the regularity of time dependent domain and state the main theorem containing the existence of the continuous density function which is proportional to the normal derivative at the boundary of the solution of the heat equation with Dirichlet boundary condition. In Section \ref{Fsec.3}, we prove the weaker form of the main theorem to apply PDE techniques and Proposition \ref{prop.2}, called the jump relation for our time varying domain that allows to have a implicit formula of the continuous density function that is an integral equation of Volterra type in Proposition \ref{prop.5}. In Section \ref{Fsec.4}, we prove the main theorem by comparing between the solution obtained by the probabilistic construction, called Feynman-Kac formula, and the one obtained by Green's formula.

\vskip1cm

\setcounter{equation}{0}

\section{Preliminaries:  problem setting and main theorem} 
\label{Fsec.2}

\ We start with the domain $\Omega$ which is a bounded connected open set in $\mathbb{R}^{d}$, $d\geq 2$. We denote by $\partial\Omega$ and $\overline{\Omega}$ the boundary of $\Omega$ and the closure of $\Omega$. $\Omega$ changes in time with respect to a continuous velocity vector field $\displaystyle v:\mathbb{R}^{d}\times{\mathbb{R}_{+}}\rightarrow \mathbb{R}^{d}$ such that there is a set of integral curves $\left\{\theta_{s}^{t}:\mathbb{R}^{d}\rightarrow\mathbb{R}^{d}\right\}_{s\leq t}$ which satisfies $\displaystyle\frac{\partial\theta_s^{t}x}{\partial t}=v(\theta_s^{t}x,t)$ and $\theta_s^{s}x=x$ for all $x\in\mathbb{R}^{d}$. For the existence of $\theta_s^{t}$, we assume that for any finite interval $I\subset\mathbb{R}_{+}$, there is $L>0$ such that $|v(x_1,t)-v(x_2,t)|\leq L |x_1-x_2|$ for all $x_1$, $x_2$ in $\mathbb{R}^{d}$ and all $t\in I$. From now on, we restrict time domain to a finite interval $[0,T]$ for fixed $T>0$. Then we have 
 \begin{prop}
 \label{prop.1}
 There is a unique homeomorphism $\theta_{s}^{t}:\mathbb{R}^{d}\rightarrow\mathbb{R}^{d}$ for any $0\leq s\leq t\leq T$, which satisfies $\displaystyle\frac{\partial\theta_s^{t}x}{\partial t}=v(\theta_s^{t}x,t)$ and $\theta_s^{s}x=x$ for all $x\in\mathbb{R}^{d}$.
 \end{prop}
 \begin{proof}
See Theorem 2.1 in Chapter 1 of \cite{Taylor} or Proposition 1.1 in Chapter 4 of \cite{Lang}.
 \end{proof}

\ Thus we can write $\displaystyle\theta_s^{t}x=x+\int_{s}^{t}v(\theta_s^{\tau}x,\tau)d\tau$ and let us write $\theta_t^{s}$ as the inverse of $\theta_s^{t}$. We denote by $\Omega_{t}$ the image of $\Omega$ under $\displaystyle\theta_0^{t}$, by abuse of notation, $\theta_s^{t}:\Omega_s\rightarrow\Omega_t$ is a homemorphism and then we can also extend the domain of $\theta_s^{t}$ into $\overline{\Omega_s}$ so that $\theta_s^{t}$ is a boundary-preserving mapping from  $\overline{\Omega_s}$  to  $\overline{\Omega_t}$.\\
Now we consider the initial-boundary value problem for the heat equation in moving domains $\{\Omega_t\}_{t\geq 0}$ as follows:
\begin{equation}
\left\{ \,
\begin{IEEEeqnarraybox}[][c]{l?s}
\IEEEstrut
\displaystyle u_t=\frac{1}{2}\Delta_x u,\ 0 < t \leq T,\ x\in\Omega_t, \\
 \displaystyle u(x,t)=0,\ 0 \leq t \leq T,\ x\in\partial\Omega_t,\\ 
  \displaystyle u(x,0)=u_0(x),\ x\in\Omega_0=\Omega.
  \IEEEstrut
\end{IEEEeqnarraybox}
\right.
\label{4242}
\end{equation}
\ To obtain a solution of \eqref{4242}, we have to find suitable conditions about the initial function $u_0$, the moving velocity $v$ and the boundaries $\partial\Omega_t$. In the next sections, we will approximate the Dirac delta function with a sequence of $C_c^{\infty}$-functions thus we assume that $u_0$ is $C_c^{\infty}$ and whose support is contained in $\Omega$. Moreover, we suppose that $\partial\Omega$ is $C^{3}$, that is, if for each point $x^{0}\in\partial\Omega$ there exist $r>0$ and a $C^{3}$ function $F:\mathbb{R}^{d-1}\rightarrow\mathbb{R}$ such that-upon relabeling and reorienting the coordinates axes if necessary-we have
\begin{eqnarray}
\displaystyle\Omega\cap\mathcal{B}(x^{0},r)=\{x\in\mathcal{B}(x^{0},r): x_d>F(x_1,\cdots,x_{d-1})\},
\end{eqnarray}
where $\mathcal{B}(x,r)=\{y\in\mathbb{R}^{d}: |x-y|<r\}$ the open ball with center $x\in\mathbb{R}^{d}$ and radius $r>0$ throughout the paper. In addition, let us assume that $v$ is $C^{3}$ so that $\theta_s^{t}x(=\theta(x,t))$ is a function of $x$ and $t$ which is $C^{3}$, thus $\theta_s^{t}:\overline{\Omega_s}\rightarrow\overline{\Omega_t}$ is $C^{3}$-diffeomorphism (c.f. Chapter $4$ of \cite{Lang}). Let us denote by the parabolic cylinder and the lateral boundary 
\begin{eqnarray}
\displaystyle D_{T}:=\bigcup_{0<t\leq T} \Omega_t\times\{ t \},\ \ S_{T}:=\bigcup_{0\leq t\leq T}\partial\Omega_t\times\{ t \}.
\end{eqnarray}
Then we have the following existence theorem.
\begin{thm}[\emph{existence}]
\label{thm1}
There exists a unique solution $u$ of \eqref{4242} such that $\displaystyle u\in C^{2,1}_{x,t}(\overline{D_T})$.
\end{thm}
\begin{proof}
See Theorem 7 in Chapter 3 of \cite{FR}.
\end{proof}
\begin{rem}
\label{rem.1}
The assumption that $\partial\Omega$ and $v$ are $C^{3}$ is to have $\nabla_x u(\cdot,t)$ is bounded on $\Omega_t$ so that it can be continuously extended to $\overline{\Omega_t}$ which is essential in Theorem \ref{mainthm} below. For the proof of Proposition \ref{prop.2} and \ref{prop.3}, it is enough that $\partial\Omega$ is $C^{2}$ and $v$ is $C^{2}$ with respect to spatial variable so that $\theta_s^{t}$ is $C^{2}$-diffeomorphism.
\end{rem}
\ Let us call $P_{r,s}, r\in\mathbb{R}, s\geq 0$, the law on $\displaystyle C([s,\infty))$ of standard $d$-dimensional  Brownian motion $B_t$, $t\geq s$, which starts from $r$ at time $s$, i.e. $B_s=r$. For each $t>s$, the law of $B_t$ is absolutely continuous with respect to the Lebesgue measure and has a density $G_{s,t}(r,\cdot)$ which is the Gaussian $\displaystyle G(\cdot,t;r,s)=\left(\frac{1}{\sqrt{2\pi (t-s)}}\right)^{d}\exp\left\{-\frac{(\cdot-r)^{2}}{2(t-s)}\right\}$. 
For $s\geq 0$ and $r\in \Omega_s$, we define
\begin{eqnarray}
\displaystyle \tau_{r,s}^{\Omega,v}:=\inf\{t \geq s: B_t\in \partial\Omega_t \},\ and=\infty\ \textrm{if the set is empty},
\end{eqnarray}
where $B_s=r$ and denote by $dF_{r,s}(y,q)$, $y\in\partial\Omega_q$, the distribution of $\displaystyle \tau_{r,s}^{\Omega,v}$ induced by $P_{r,s}$. For $s=0$, we use abbreviated forms  $P_{r}$, $E_{r}$, $\displaystyle \tau_{r}$, $dF_{r}(y,q)$  instead of $P_{r,0}$, $E_{r,0}$, $\displaystyle \tau_{r,0}^{\Omega,v}$, $dF_{r,0}(y,q)$ respectively whenever it is needed.
In addition, for $r_0\in\Omega$ and $t>0$, let us call $d\mu_{r_0}(\cdot,t)$ the positive measure on $\Omega_t$ such that 
\begin{eqnarray}
\label{1.0}
\displaystyle \int_{\Omega_t} d\mu_{r_0}(x,t)f(x)=E_{r_0}[f(B_t);\tau_{r_0}^{\Omega,v}\geq t]
\end{eqnarray}
for all $\displaystyle f\in C_c^{\infty}(\mathbb{R}^d)$ with $\displaystyle\textrm{supp}\hspace{0.5mm}f\Subset \Omega_t$.

\ Since $\partial\Omega_t$ is $C^{3}$, we have the outward pointing unit normal vector field $n=n_{x,t}$ at $x\in\partial\Omega_t$ and let us denote $\dfrac{\partial f}{\partial n}(x,t):=n\cdot \nabla_x f(x,t)$ for $f\in C_x^{1}(\overline{D_T})$. Throughout the paper, we write $d\mathcal H^{d-1}$ as the $(d-1)$-dimensional Hausdorff measure on $\mathbb{R}^{d}$. The main result in the paper is;
\begin{thm} 
\label{mainthm}
Under the same assumption as in Theorem \ref{thm1}, we have for any $r_0\in\Omega_0=\Omega$, 
\begin{enumerate}
\item\label{1.1} $d\mu_{r_0}(x,t)=G_{0,t}^{\Omega, v}(r_0,x)dx$ where for all $x\in\Omega_t$,
\begin{eqnarray}
\label{12345}
\displaystyle G_{0,t}^{\Omega,v}(r_0,x)=G_{0,t}(r_0,x)-\int_{[0,t)}\int_{\partial\Omega_s}G_{s,t}(y,x)dF_{r_0}(y,s).
\end{eqnarray}
\item\label{1.2} $\displaystyle dF_{r_0}(y,s)$ has a continuous density function $p$ such that $dF_{r_0}(y,s)=p(y,s)d\mathcal{H}^{d-1}(y)ds$.\\
\item\label{1.3} $\displaystyle p(x,t)=-\frac{1}{2}\frac{\partial}{\partial n}G_{0,t}^{\Omega,v}(r_0,x)$ for all\ $t>0$ and $x\in\partial\Omega_t$.\\
\item\label{1.4} $\displaystyle G_{0,t}^{\Omega,v}(r_0,x)$ solves 
\begin{eqnarray} 
\label{1}\displaystyle w_t=\frac{1}{2}\Delta_{x}w,\ \ x\in \Omega_t,\ t>0, \\
\label{2}\displaystyle \lim_{\Omega_t\ni x\rightarrow y}w(x,t)=0,\  y\in\partial\Omega_t, t>0,\\
\label{3} \displaystyle \lim_{(x,t)\rightarrow (y,0)}w(x,t)=\delta_{r_0}(y),\ y\in\Omega_0.
\end{eqnarray}
\end{enumerate}
\end{thm}
\vskip 1cm

\section{Weaker Form of Theorem \ref{mainthm}}
\label{Fsec.3}
\ Let us first prove item \ref{1.1} of Theorem \ref{mainthm}. By \eqref{1.0} and the strong Markov property of Brownian motion,
\begin{eqnarray}
&&\displaystyle \int_{\Omega_t} d\mu_{r_0}(x,t)f(x)=E_{r_0}[f(B_t);\tau_{r_0}^{\Omega,v}\geq t]=E_{r_0}[f(B_t)]-E_{r_0}[f(B_t);\tau_{r_0}^{\Omega,v}< t]\nonumber\\
&&=\int_{\Omega_t}f(x)G_{0,t}(r_0,x)dx-\int_{[0,t)}\int_{\partial\Omega_s}E_{y,s}[f(B_t)|B_s=y, \tau_{r_0}^{\Omega,v}=s]dF_{r_0}(y,s)\nonumber\\
&&=\int_{\Omega_t}f(x)G_{0,t}(r_0,x)dx-\int_{\Omega_t}f(x)\int_{[0,t)}\int_{\partial\Omega_s}G_{s,t}(y,x)dF_{r_0}(y,s)dx\nonumber\\
&&=\int_{\Omega_t}f(x)\left(G_{0,t}(r_0,x)-\int_{[0,t)}\int_{\partial\Omega_s}G_{s,t}(y,x)dF_{r_0}(y,s)\right)dx.
\end{eqnarray}
Thus the proof is done. Before proving the other items of Theorem \ref{mainthm}, we study the weaker form of it that we will specify below because the initial datum of item \ref{1.4} is the Dirac delta function which is hard to control directly so that we let it as $C_c^{\infty}$ and apply approximation as usual. By item \ref{1.1} of Theorem \ref{mainthm}, we have $G_{0,t}^{\Omega,v}$ so let us define
\begin{eqnarray}
\label{2345}
\displaystyle u(x,t):=\int_{\Omega_0=\Omega}u_0(\xi)G_{0,t}^{\Omega,v}(\xi,x)d\xi
\end{eqnarray}
for given $\displaystyle u_0\in C_c^{\infty}(\Omega;\mathbb{R}_{+})$ and all $(x,t)\in D$.

\ We prove the following weaker form of Theorem \ref{mainthm}.
\begin{thm}
\label{thm3}
Under the same assumption as in Theorem \ref{thm1}, we have
\begin{enumerate}
\item\label{thm1.1}The function $u$ defined in $(\ref{2345})$ is the unique solution of Theorem \ref{thm1}.\\
\item\label{thm1.2}Moreover, for all $t > 0$, $\displaystyle p_{u_0}(x,t):=-\frac{1}{2}\frac{\partial u}{\partial n}(x,t)$ satisfies
\begin{eqnarray}
\label{888}
\displaystyle p_{u_0}(x,t)=-\int_{\Omega_0}u_0(\xi)\frac{\partial G}{\partial n_{x,t}}(x,t;\xi,0)d\xi+\int_{0}^{t}\int_{\partial\Omega_s}\frac{\partial G}{\partial n_{x,t}}(x,t;y,s)p_{u_0}(y,s)d\mathcal{H}^{d-1}(y)ds.\nonumber\\
\end{eqnarray}

\end{enumerate}

\end{thm}


\ Before going to the proof of Theorem \ref{thm3}, we will prove Proposition \ref{prop.2} and Proposition \ref{prop.3} below, as mentioned in Remark $\ref{rem.1}$, we assume that $\partial\Omega$ is $C^{2}$ and $v$ is $C^{2}$ with respect to spatial variable. First the following Lemma \ref{lem.1} and \ref{lem.2} are needed for Proposition \ref{prop.2}.

\begin{lem}
\label{lem.1}
There is $C>0$ such that for all $0\leq t\leq T$ and all $x\in\partial\Omega_t$, we have
\begin{eqnarray}
\displaystyle |\langle y-x,n_{x,t}\rangle|\leq C|y-x|^2.
\end{eqnarray}
if $y\in\partial\Omega_t$ is sufficiently close to $x$.
\end{lem}
\begin{proof}
First we fix $0\leq t\leq T$ and $x\in\partial\Omega_t$. There is a local representation $F\in C^{2}$ at $x$, without loss of generality, we may assume that $x=0$ and $DF(x)=0$ such that all $y\in\partial\Omega_t$ sufficiently close to $x$ can be written as $y=(y_1,\cdots,y_d)=(y_1,  \cdots, y_{d-1},F(y_1,\cdots, y_{d-1}))$. By Taylor's theorem, we have
\begin{eqnarray}
\displaystyle |\langle y-x,n_{x,t}\rangle|=|\langle y,n_{x,t}\rangle|=|F(y_1,\cdots,y_{d-1})|\leq\lVert D^{2}F \rVert_{\infty} \sum_{i=1}^{d-1}|y_i|^{2}\leq \lVert D^{2}F \rVert_{\infty}|y|^2.
\end{eqnarray}
Thus the boundaries $\left\{\partial\Omega_t\right\}_{0\leq t\leq T}$ are compact and diffeomorphic to each other under $\theta$ so that the second differential of a local representation is uniformly bounded. The proof is complete.
\end{proof}
\begin{lem}
\label{lem.2}
Given $\epsilon>0$, there is $\delta>0$ such that for all $0<t\leq T$ and all $x\in\partial\Omega_t$,
\begin{eqnarray}
\displaystyle \left|\int_{\partial\Omega_s}\left(\frac{1}{\sqrt{2\pi(t-s)}}\right)^{d-1}\exp\left\{-\frac{|x-y|^2}{2(t-s)}\right\}d\mathcal{H}^{d-1}(y)-1\right|<\epsilon
\end{eqnarray}
if $t-\delta < s < t$.
\end{lem}
\begin{proof}
We fix $0< t\leq T$ and $x\in\partial\Omega_t$. For a local representation $F\in C^{2}$ at $x$, without loss of generality, we may assume that $x=0$ and $DF(x)=0$ such that all $y\in\partial\Omega_t$ sufficiently close to $x$ can be written as $y=(y_1,\cdots,y_d)=(y_1,  \cdots, y_{d-1},F(y_1,\cdots, y_{d-1}))$. Let us choose $\gamma$ and $\eta$ such that $\frac{1}{4}<\gamma<\frac{1}{2}$ and $\displaystyle\eta>\sup_{\overline{D_T}}|v|$. We define $E_s^{1}:= \partial\Omega_s\cap \mathcal{B}(x,\eta(t-s)^{\gamma}),\ E_s^{2}:= \partial\Omega_s - E_s^{1}$.  Then $\theta_{t}^{s}x\in E_s^{1}$ for all $s$ sufficiently close to $t$.
By a change of variables, we obtain
\begin{eqnarray}
\label{3.17}
\displaystyle I&=&\int_{E_s^{1}}\left(\frac{1}{\sqrt{2\pi(t-s)}}\right)^{d-1} \exp\left\{-\frac{|\theta_s^{t}y|^2}{2(t-s)}\right\}  \left|\textrm{Jac}(\theta_s^{t}y)\right| d\mathcal{H}^{d-1}(y)\nonumber\\
\displaystyle&=&\int_{\theta_s^{t}(E_s^{1})}\left(\frac{1}{\sqrt{2\pi(t-s)}}\right)^{d-1}\exp\left\{-\frac{|y|^2}{2(t-s)}\right\}d\mathcal{H}^{d-1}(y)\nonumber\\
\displaystyle&=&\int_{P(\theta_s^{t}(E_s^{1}))}\left(\frac{1}{\sqrt{2\pi(t-s)}}\right)^{d-1}\exp\left\{-\frac{ \sum\limits_{i=1}^{d-1}|y_i|^{2}+|F(y_1,\cdots, y_{d-1})|^2}{2(t-s)}\right\}\sqrt{|A|}dy_1\cdots dy_{d-1},\nonumber\\
\end{eqnarray}
where Jac is a Jacobian matrix and $P$ is a projection such that $P(y_1,\cdots,y_d)=(y_1,\cdots,y_{d-1})$ and $A$ is a $(d-1)\times(d-1)$-matrix whose element $a_{ij}$ is given by $\delta_{ij}+\frac{\partial F}{\partial y_i}\frac{\partial F}{\partial y_j}$.\\
If $y\in E_s^{1}$, then $|\theta_s^{t}y|^2=|y|^2+2|\langle y,\int_s^{t}v(\theta_s^{\tau}y,\tau)d\tau\rangle|+|\int_s^{t}v(\theta_s^{\tau}y,\tau)d\tau|^2\leq |y|^2+2\eta^{2}(t-s)^{1+\gamma}+\eta^{2}(t-s)^2$ and $|\textrm{Jac}(\theta_s^{t}y)|=\left|\left[\delta_{ij}+\int_{s}^{t}\frac{\partial}{\partial y_j}\left(v^{(i)}(\theta_s^{\tau}y,\tau)\right)d\tau\right]\right|$. Thus for given $\epsilon>0$, there is $\delta_1>0$ which does not depend on the choice of $(x,t)$ such that for all $t-\delta_1<s<t$,
\begin{eqnarray}
\displaystyle \left|I-\int_{E_s^{1}}\left(\frac{1}{\sqrt{2\pi(t-s)}}\right)^{d-1} \exp\left\{-\frac{|y|^2}{2(t-s)}\right\} d\mathcal{H}^{d-1}(y)\right|<\frac{\epsilon}{4}.
\end{eqnarray}

If $(y_1,\cdots,y_{d-1})\in P(\theta_s^{t}(E_s^{1}))$, then $|\theta_t^{s}(y_1,\cdots,y_{d-1}, F(y_1,\cdots,y_{d-1}))|\leq \eta(t-s)^{\gamma}$. So we have $|(y_1,\cdots,y_{d-1})|\leq 2\eta(t-s)^{\gamma}$ for all $s$ sufficiently close to $t$, thus 
\begin{eqnarray}
|F(y_1,\cdots,y_{d-1})|\leq\lVert D^{2}F \rVert_{\infty}4\eta^{2}(t-s)^{2\gamma}
\end{eqnarray}
and then
\begin{eqnarray}
\displaystyle-\frac{|F(y_1,\cdots,y_{d-1})|^2}{2(t-s)}\geq -8\eta^{4}\lVert D^{2}F \rVert_{\infty}^{2}(t-s)^{4\gamma-1}.
\end{eqnarray}
By the mean value theorem, we get
\begin{eqnarray}
\displaystyle\left|\frac{\partial F}{\partial y_i}\right|\leq \lVert D^{2}F \rVert_{\infty}|(y_1,\cdots,y_{d-1})|\leq\lVert D^{2}F \rVert_{\infty}2\eta(t-s)^{\gamma}.
\end{eqnarray}
Since the second differential of a local representation is uniformly bounded, for given $\epsilon>0$, there is $\delta_2>0$ which does not depend on the choice of $(x,t)$ such that for all $t-\delta_2<s<t$,
 
\begin{eqnarray}
\displaystyle \left|I-\int_{P(\theta_s^{t}(E_s^{1}))}\left(\frac{1}{\sqrt{2\pi(t-s)}}\right)^{d-1}\exp\left\{-\frac{ \sum\limits_{i=1}^{d-1}|y_i|^{2}}{2(t-s)}\right\}dy_1\cdots dy_{d-1}\right|<\frac{\epsilon}{4}.
\end{eqnarray}

Moreover, if $\displaystyle |(y_1,\cdots,y_{d-1})|\leq \frac{\eta}{2}(t-s)^{\gamma}$, then $(y_1,\cdots,y_{d-1})\in P(\theta_s^{t}(E_s^{1}))$ for all $s$ sufficiently close to $t$, so for given $\epsilon>0$, there is $\delta_3>0$ which does not depend on the choice of $(x,t)$ such that for all $t-\delta_3<s<t$,
\begin{eqnarray}
\displaystyle \left|\int_{P(\theta_s^{t}(E_s^{1}))}\left(\frac{1}{\sqrt{2\pi(t-s)}}\right)^{d-1}\exp\left\{-\frac{ \sum\limits_{i=1}^{d-1}|y_i|^{2}}{2(t-s)}\right\}dy_1\cdots dy_{d-1}-1\right|<\frac{\epsilon}{4}
\end{eqnarray}
If we take $\delta_4=\min\{\delta_1,\delta_2,\delta_3\}$, we conclude that for given $\epsilon>0$, there is $\delta_4>0$ which does not depend on the choice of $(x,t)$ such that for all $t-\delta_4<s<t$,
\begin{eqnarray}
&&\displaystyle\left|\int_{E_s^{1}}\left(\frac{1}{\sqrt{2\pi(t-s)}}\right)^{d-1} \exp\left\{-\frac{|y|^2}{2(t-s)}\right\} d\mathcal{H}^{d-1}(y)-1\right|<\frac{3\epsilon}{4}.
\end{eqnarray}
For all $y\in E_s^{2}$, we have
\begin{eqnarray}
\displaystyle\left(\frac{1}{\sqrt{2\pi(t-s)}}\right)^{d-1} \exp\left\{-\frac{|y|^2}{2(t-s)}\right\}\leq \left(\frac{1}{\sqrt{2\pi(t-s)}}\right)^{d-1} \exp\left\{-\frac{\eta^{2}}{2(t-s)^{1-2\gamma}}\right\}
\end{eqnarray}
and so for given $\epsilon>0$, there is $\delta_5>0$ which does not depend on the choice of $(x,t)$ such that for all $t-\delta_5<s<t$,
\begin{eqnarray}
\displaystyle \left|\int_{E_s^{2}}\left(\frac{1}{\sqrt{2\pi(t-s)}}\right)^{d-1} \exp\left\{-\frac{|y|^2}{2(t-s)}\right\}d\mathcal{H}^{d-1}(y)\right|<\frac{\epsilon}{4}.
\end{eqnarray}
Therefore if we let $\delta$ be the minimum value of $\delta_i$'s, the proof is complete.
\end{proof}
\vskip 0.5cm
\ The following Proposition \ref{prop.2}, we call it a jump relation, is the most significant property of the single-layer potential, whose one dimensional version is proved in \cite{Cannon} and also applied in the analysis of \cite{JMLee}. We develope it into a multidimensional case under the assumption discussed above and give a rigorous calculation.
\begin{prop}[\emph{jump relation}]
\label{prop.2}
For $\displaystyle\varphi\in C(S_T)$, we have
\begin{eqnarray}
\label{327}
&&\displaystyle\lim_{h\rightarrow 0^{+}}\int_{0}^{t}\int_{\partial \Omega_s}n_{x,t}\cdot \nabla_{x} G(x-hn_{x,t},t;y,s)\varphi(y,s)d\mathcal{H}^{d-1}(y)ds\nonumber\\
&&\hspace{2cm}=\varphi(x,t)+\int_{0}^{t}\int_{\partial \Omega_s}\frac{\partial G}{\partial n_{x,t}}(x,t;y,s)\varphi(y,s)d\mathcal{H}^{d-1}(y)ds,
\end{eqnarray}
for all $0< t\leq T$ and $x\in \partial\Omega_t$.
\end{prop}
\begin{proof}
Let us fix  $0< t\leq T$ , $x\in \partial\Omega_t$. Let $\gamma$, $\eta$,  $E_s^{1}$,  $E_s^{2}$ be same as in the proof of Lemma \ref{lem.2}. Without loss of generality, we may assume that $\varphi \geq 0$. We can write 
\begin{eqnarray}
\label{328}
\int_{0}^{t}\int_{\partial \Omega_s}n_{x,t}\cdot \nabla_{x}G(x-hn_{x,t},t;y,s)\varphi(y,s)d\mathcal{H}^{d-1}(y)ds = I_1+I_2+I_3+I_4
\end{eqnarray}
where
\begin{eqnarray*}
&&\displaystyle I_1=\int_{t-\delta}^{t}\int_{E_s^{1}}\frac{h}{t-s}\left(\frac{1}{\sqrt{2\pi(t-s)}}\right)^{d}\exp\left\{-\frac{|x-hn_{x,t}-y|^2}{2(t-s)}\right\}\varphi(y,s)d\mathcal{H}^{d-1}(y)ds,\\
&&\displaystyle I_2=\int_{t-\delta}^{t}\int_{E_s^{1}}\frac{\langle y-x,n_{x,t}\rangle}{t-s} \left(\frac{1}{\sqrt{2\pi(t-s)}}\right)^{d}\exp\left\{-\frac{|x-hn_{x,t}-y|^2}{2(t-s)}\right\}\varphi(y,s)d\mathcal{H}^{d-1}(y)ds,\\
&&\displaystyle I_3=\int_{t-\delta}^{t}\int_{E_s^{2}}\frac{\langle y-x+hn_{x,t},n_{x,t}\rangle}{t-s}  \left(\frac{1}{\sqrt{2\pi(t-s)}}\right)^{d}\exp\left\{-\frac{|x-hn_{x,t}-y|^2}{2(t-s)}\right\}\varphi(y,s)d\mathcal{H}^{d-1}(y)ds,\\
&&\displaystyle I_4=\int_{0}^{t-\delta}\int_{\partial \Omega_s}n_{x,t}\cdot \nabla_{x}G(x-hn_{x,t},t;y,s)\varphi(y,s)d\mathcal{H}^{d-1}(y)ds.
\end{eqnarray*}

For $I_1$, we can rewrite as follows: 
\begin{eqnarray*}
&&\displaystyle J(s,h)=\int_{E_s^{1}} \left( \frac{1}{\sqrt{2\pi(t-s)}}\right)^{d-1}\exp\left\{-\frac{|x-y|^2}{2(t-s)}\right\}\exp\left\{\frac{h\langle x-y ,n_{x,t}\rangle}{t-s}\right\}\varphi(y,s)d\mathcal{H}^{d-1}(y),\\
&&\displaystyle I_1 = \int_{t-\delta}^{t}\frac{h}{t-s}\frac{1}{\sqrt{2\pi(t-s)}}\exp\left\{-\frac{h^2}{2(t-s)}\right\}J(s,h)ds.
\end{eqnarray*}
By Lemma \ref{lem.1}, there is $C>0$ such that
\begin{eqnarray}
\label{329}
\displaystyle \frac{|\langle x-y,n_{x,t}\rangle|}{t-s}\leq \frac{|\langle x-\theta_s^{t}y+\theta_s^{t}y-y,n_{x,t}\rangle|}{t-s}\leq \frac{C}{(t-s)^{1-2\gamma}}.
\end{eqnarray}
for all $s$ sufficiently close to $t$ and all $y\in E_s^{1}$. 

Then by \eqref{329},
\begin{eqnarray}
\label{330}
 J(s,0)\exp\left\{-\frac{Ch}{(t-s)^{1-2\gamma}}\right\}\leq J(s,h)\leq J(s,0)\exp\left\{\frac{Ch}{(t-s)^{1-2\gamma}}\right\}.
\end{eqnarray}
By \eqref{330} and a change of variable as $z=\frac{h}{\sqrt{t-s}}$, we have
\begin{eqnarray}
\label{331}
&&\displaystyle \int_{\frac{h}{\sqrt{\delta}}}^{\infty}\frac{2}{\sqrt{2\pi}}\exp\left\{-\frac{|z|^2}{2}\right\}\exp\left\{-Cz(t-s)^{2\gamma-\frac{1}{2}}\right\}J(t-\frac{h^2}{z^2},0)dz\leq I_1\nonumber\\
&&\hspace{1cm} \leq \int_{\frac{h}{\sqrt{\delta}}}^{\infty}\frac{2}{\sqrt{2\pi}}\exp\left\{-\frac{|z|^2}{2}\right\}\exp\left\{Cz(t-s)^{2\gamma-\frac{1}{2}}\right\}J(t-\frac{h^2}{z^2},0)dz.
\end{eqnarray}
By Lemma \ref{lem.2}, given $\epsilon>0$, there is $\delta>0$ such that 
\begin{eqnarray}
\label{332}
|J(s,0)-\varphi(x,t)|<\epsilon
\end{eqnarray}
 for all $t-\delta<s<t$. Therefore, it follows by \eqref{331} and \eqref{332} that
\begin{eqnarray}
&&\displaystyle \int_{\frac{h}{\sqrt{\delta}}}^{\infty}\frac{2}{\sqrt{2\pi}}\exp\left\{-\frac{|z+C\delta^{2\gamma-\frac{1}{2}}|^2}{2}\right\}\exp\left\{\frac{C^2\delta^{4\gamma-1}}{2}\right\}(\varphi(x,t)-\epsilon)dz\leq I_1\nonumber\\
&&\hspace{0.5cm} \leq \int_{\frac{h}{\sqrt{\delta}}}^{\infty}\frac{2}{\sqrt{2\pi}}\exp\left\{-\frac{|z-C\delta^{2\gamma-\frac{1}{2}}|^2}{2}\right\}\exp\left\{\frac{C^2\delta^{4\gamma-1}}{2}\right\}(\varphi(x,t)+\epsilon)dz
\end{eqnarray}
and so we have
\begin{eqnarray}
\displaystyle\left|\lim_{\delta\rightarrow 0^{+}}\lim_{h\rightarrow 0^{+}}I_1-\varphi(x,t)\right|<\epsilon.
\end{eqnarray}
For $I_2$, by \eqref{329}, \eqref{330} and \eqref{332}, we have
\begin{eqnarray}
\displaystyle |I_2|\leq \int_{t-\delta}^{t}\frac{C_1}{(t-s)^{\frac{3}{2}-2\gamma}}\exp\left\{\frac{C^2\delta^{4\gamma-1}}{2} \right\}(\varphi(x,t)+\epsilon)ds\leq C_2\delta^{2\gamma-\frac{1}{2}}
\end{eqnarray}
and so it follows that $\displaystyle\lim_{\delta\rightarrow 0^{+}}\lim_{h\rightarrow 0^{+}}|I_{2}|=0$.\\
For $I_3$, all sufficiently $h>0$ such that $\left(\frac{2h}{\eta}\right)^{\frac{1}{\gamma}}<\delta$, we can decompose $I_3$ into $I_{3,1}+I_{3,2}$ as
\begin{eqnarray*}
\displaystyle I_{3,1}=\int_{t-\delta}^{t-\left(\frac{2h}{\eta}\right)^{\frac{1}{\gamma}}}\int_{E_s^{2}}\frac{\langle y-x+hn_{x,t},n_{x,t}\rangle}{t-s}  \left(\frac{1}{\sqrt{2\pi(t-s)}}\right)^{d}\exp\left\{-\frac{|x-hn_{x,t}-y|^2}{2(t-s)}\right\}\varphi(y,s)d\mathcal{H}^{d-1}(y)ds,\\
\displaystyle I_{3,2}=\int_{t-\left(\frac{2h}{\eta}\right)^{\frac{1}{\gamma}}}^{t}\int_{E_s^{2}}\frac{\langle y-x+hn_{x,t},n_{x,t}\rangle}{t-s}  \left(\frac{1}{\sqrt{2\pi(t-s)}}\right)^{d}\exp\left\{-\frac{|x-hn_{x,t}-y|^2}{2(t-s)}\right\}\varphi(y,s)d\mathcal{H}^{d-1}(y)ds.
\end{eqnarray*}

If $y\in E_s^{2}$ and $t-\delta<s<t-\left(\frac{2h}{\eta}\right)^{\frac{1}{\gamma}}$, then $|x-y|\geq\eta(t-s)^{\gamma}>2h$, so we have
\begin{eqnarray*}
&&\displaystyle|I_{3,1}|\leq C_3\int_{t-\delta}^{t-\left(\frac{2h}{\eta}\right)^{\frac{1}{\gamma}}}\frac{1}{(t-s)^{\frac{d}{2}}}\int_{E_s^{2}}\frac{|x-y|}{t-s}\exp\left\{-\frac{|x-y|^2}{8(t-s)}   \right\}d\mathcal{H}^{d-1}(y)ds\\
&&\hspace{1cm}\leq C_3\int_{t-\delta}^{t-\left(\frac{2h}{\eta}\right)^{\frac{1}{\gamma}}}\frac{1}{\eta(t-s)^{\gamma+\frac{d}{2}}}\int_{E_s^{2}}\frac{|x-y|^2}{t-s}\exp\left\{-\frac{|x-y|^2}{8(t-s)}   \right\}d\mathcal{H}^{d-1}(y)ds\\
&&\hspace{1cm}\leq C_4\int_{t-\delta}^{t-\left(\frac{2h}{\eta}\right)^{\frac{1}{\gamma}}}\int_{E_s^{2}}\frac{1}{(t-s)^{1-\gamma+\frac{d}{2} } }\exp\left\{-\frac{C_5}{(t-s)^{1-2\gamma}} \right\}d\mathcal{H}^{d-1}(y)ds
\end{eqnarray*}
and so it follows that $\displaystyle\lim_{\delta\rightarrow 0^{+}}\lim_{h\rightarrow 0^{+}}|I_{3,1}|=0$.\\
Let $y\in \partial \Omega_s$. For all sufficiently small $h>0$, if $t-\left(\frac{2h}{\eta}\right)^{\frac{1}{\gamma}}\leq s\leq t$ and $\gamma<\beta<\frac{1}{2}$, then 
\begin{eqnarray*}
&&\displaystyle |x-hn_{x,t}-y|\geq |x-hn_{x,t}-\theta_s^{t}y|-|\theta_s^{t}y-y|\geq d(x-hn_{x,t},\partial\Omega_t)-\eta(t-s)\\
&&\hspace{2cm}\geq h-C_6h^{2}-\eta(t-s)\geq (t-s)^{\beta},
\end{eqnarray*}
where $d(x-hn_{x,t},\partial\Omega_t)=\inf\{|x-hn_{x,t}-y|:y\in\partial\Omega_t\} \geq h-C_6h^{2}$ is by Lemma \ref{lem.1}.\\ 
Hence we get
\begin{eqnarray*}
&&|I_{3,2}|\\
&&\leq C_7\int_{t-\left(\frac{2h}{\eta}\right)^{\frac{1}{\gamma}}}^{t} \frac{1}{(t-s)^{\beta+\frac{d}{2}}}\int_{E_s^{2}}\frac{|x-hn_{x,t}-y|^2}{t-s}\exp\left\{-\frac{|x-hn_{x,t}-y|^2}{2(t-s)}   \right\}d\mathcal{H}^{d-1}(y)ds\\
&&\leq C_8\int_{t-\left(\frac{2h}{\eta}\right)^{\frac{1}{\gamma}}}^{t} \int_{E_s^{2}}\frac{1}{(t-s)^{1-\beta+\frac{d}{2} } }\exp\left\{-\frac{1}{2(t-s)^{1-2\beta}} \right\}d\mathcal{H}^{d-1}(y)ds
\end{eqnarray*}
and so it follows that $\displaystyle\lim_{h\rightarrow 0^{+}}|I_{3,2}|=0$.\\
For $I_4$, by Lebesgue's dominated convergence theorem, it follows that
\begin{eqnarray*}
\lim_{h\rightarrow 0^{+}}I_4=\int_{0}^{t-\delta}\int_{\partial \Omega_s}\frac{\partial G}{\partial n_{x,t}}(x,t;y,s)\varphi(y,s)d\mathcal{H}^{d-1}(y)ds.
\end{eqnarray*}
Finally, by combining all estimates above, we obtain 
\begin{eqnarray*}
\left|\lim_{\delta\rightarrow 0^{+}}\lim_{h\rightarrow 0^{+}}(I_1+I_2+I_3+I_4)-\varphi(x,t)-\int_{0}^{t}\int_{\partial \Omega_s}\frac{\partial G}{\partial n_{x,t}}(x,t;y,s)\varphi(y,s)d\mathcal{H}^{d-1}(y)ds\right|<\epsilon.
\end{eqnarray*}
Since $\epsilon$ is arbitrary, so the proof is complete.
\end{proof}
\begin{rem}
In Chapter 5 of \cite{FR}, the jump relation is proved for time independent domains and in Chapter 3 of \cite{LM}, it is also proved when the domain $D$ is given by $D=\{(Z,t): z_d>f(z_1,\cdots,z_{d-1},t)\}$ where $f:\mathbb{R}^{d}\rightarrow\mathbb{R}$ satisfies
\begin{eqnarray*}
&&|f(x,t)-f(y,t)|\leq a_1|x-y|,\ \ x,\ y\in\mathbb{R}^{d-1},\ t\in\mathbb{R},\\
&&\hspace{0.5cm}f(x,t)=I_{\frac{1}{2}}(b(x,\cdot))(t)=\int_{\mathbb{R}}|s-t|^{-\frac{1}{2}}b(x,s)ds
\end{eqnarray*}
where $x\in\mathbb{R}^{d-1}$ is fixed and $b(x,\cdot)$ is of bounded mean oscillation on $\mathbb{R}$.
\end{rem}
\vskip 0.3cm
\ The following Proposition depends on a local property of Brownian motion and the regularity of the boundary whose one dimensional version is done in \cite{JMLee}, which is about the accessibility of the boundary necessary to show Dirichlet boundary condition in Theorem \ref{thm3}.
\begin{prop}
\label{prop.3}
 If the starting point of Brownian motion is close to $X$, the first passage time converges to 0. Precisely, $\displaystyle\lim_{\Omega_0\ni\xi\rightarrow\xi_0\in\partial\Omega_0}P_{\xi}\left[\tau_{\xi}^{\Omega,v}> s\right]=0$ for all $s>0$.
\end{prop}
\begin{proof}
Without loss of generality, we may assume that $0\in\partial\Omega_0$ and the tangent plane of $\partial\Omega_0$ at $0$ and $e_d$ is the outward unit normal vector at $0$. A standard $d$-dimensional Brownian motion $B_t=(B_t^{(1)},\cdots,B_t^{(d)})$ which starts at $\xi=(\xi^{(1)},\cdots,\xi^{(d)})$. Since $\theta_t^{0}$ is a diffeomorphim, $B_t$ is out of $\Omega_t$ if and only if $\displaystyle\theta_t^{0} B_t=B_t-\int_{0}^{t}v(\theta_t^{s}B_t,s)ds$ is out of $\Omega_0$. We have $a>0$ such that $\displaystyle\mathcal{B}(0,a)\cap\partial\Omega_0 = \{x\in\mathbb{R}^{d}: x^{(d)}=F(x^{(1)},\cdots,x^{(d-1)})\} $ for some $F\in C^{2}$. Thus there is $c>0$ which does not depend on $a$ such that if $x\in\mathcal{B}(0,a)\cap\partial\Omega_0$, then $|x^{(d)}|\leq ca^2$. Since $\sup |v| <\eta$, we have
\begin{eqnarray}
\displaystyle\inf\{t\geq 0: (\theta_t^{0}B_t)^{(d)}=ca^2\}\leq \inf\{t\geq 0: B_t^{(d)}-\eta t=ca^{2}\}
\end{eqnarray}
and
\begin{eqnarray}
\displaystyle\inf\left\{t\geq 0: \left|((\theta_t^{0}B_t)^{(1)},\cdots,(\theta_t^{0}B_t)^{(d-1)})\right|=a\right\}&\geq&\inf\left\{t\geq 0: \left|(B_t^{(1)},\cdots,B_t^{(d-1)})\right|=a-\eta t\right\}\nonumber\\
&\geq& U_{\frac{a-\eta t}{\sqrt{d-1}}}^{(1)}\wedge\cdots\wedge U_{\frac{a-\eta t}{\sqrt{d-1}}}^{(d-1)},
\end{eqnarray}
where $\displaystyle U_{\frac{a-\eta t}{\sqrt{d-1}}}^{(i)}=\inf \left\{ t\geq 0: \left|B_t^{(i)}\right|=\frac{a-\eta t}{\sqrt{d-1}}\right\}$ for $1\leq i \leq d-1$. It satisfies that $\dfrac{a}{\sqrt{d-1}} > ca^2$ for all sufficiently small $a>0$  and for such fixed $a$, we have$\dfrac{a}{\sqrt{d-1}}-|\xi_i|\geq ca^2-\xi_d$ for all $\xi$ sufficiently close to $0$ and all $1\leq i\leq d-1$. 

For any standard one dimensional Brownian motion $\tilde{B}_t$ which starts at $0$, we have $\displaystyle\limsup_{t\downarrow 0}\frac{\tilde{B}_t}{\sqrt{t}}=\infty$, then there is a sequence ${t_k}\downarrow0$ such that $2\eta t_k\leq \tilde{B}_{t_k}$. Thus if we fix $a$ sufficiently small, then we obtain for all $\xi$ sufficiently close to $0$,
\begin{eqnarray}
\inf\left\{t\geq 0: \tilde{B}_t=ca^2+\eta t-\xi_d\right\}\leq \inf\left\{t\geq 0: \tilde{B}_t=\frac{a-\eta t}{\sqrt{d-1}}-\xi_i \right\}.
\end{eqnarray}
and since $\eta\geq\frac{\eta}{\sqrt{d-1}}$, we also get
\begin{eqnarray}
\inf\left\{t\geq 0: \tilde{B}_t=ca^2+\eta t-\xi_d\right\}\leq \inf\left\{t\geq 0: \tilde{B}_t=\frac{\eta t-a}{\sqrt{d-1}}-\xi_i \right\}.
\end{eqnarray}
 $\displaystyle U_{\frac{a-\eta t}{\sqrt{d-1}}}^{(i)}$ has the same law with $\inf\left\{t\geq 0: \tilde{B}_t=\frac{a-\eta t}{\sqrt{d-1}}-\xi_i \right\}\wedge\inf\left\{t\geq 0: \tilde{B}_t=\frac{\eta t-a}{\sqrt{d-1}}-\xi_i \right\}$ and also $\inf\left\{t\geq 0: \tilde{B}_t=ca^2+\eta t-\xi_d\right\}$ has the same law with $\inf\{t\geq 0: B_t^{(d)}-\eta t=ca^{2}\}$, therefore, we conclude that
 \begin{eqnarray}
 \displaystyle\inf\{t\geq 0: (\theta_t^{0}B_t)^{(d)}=ca^2\}\leq\inf\left\{t\geq 0: \left|(\theta_t^{0}B_t)^{(1)},\cdots,(\theta_t^{0}B_t)^{(d-1)})\right|=a\right\}.
 \end{eqnarray}
Since we can fix $a$ arbitrary small so that for any $s>0$, the first passage time can be smaller than $s$ as the starting point $\xi$ is sufficiently close to $0$. The proof is complete.
\end{proof}

\begin{proof}\hspace{-1mm}
\textbf{of Theorem~\ref{thm3}}\\
Using the invariance of the law of the Brownian motion under time reversal, we have
\begin{eqnarray}
\displaystyle u(x,t)=\int_{\Omega_0} u_0(\xi)G_{0,t}^{\Omega,v}(\xi,x)d\xi=E_{x}[u_0(B_t);\tau_x^{\Omega_t, -v}\geq t],
\end{eqnarray}
where $\tau^{\Omega_t, -v}$ denote the first passage time when the time varying domain starts with $\Omega_t$ and changes with respect to $-v$. Using this equality, we also have
\begin{eqnarray}
\displaystyle\left|u(x,t)\right|=\left|E_{x}[u_0(B_t);\tau_x^{\Omega_t,-v}\geq t]\right|\leq \lVert u_0\rVert_{\infty}P_{x}[\tau_x^{\Omega_t,-v}\geq t].
\end{eqnarray} 
Thus by Proposition \ref{prop.2}, we have $u(x,t)\rightarrow 0$ as $x$ approaches to $\partial\Omega_t$.\\\
Let us prove that $u$ satisfies the initial data $u_0$, that is, $\displaystyle\lim_{(x,t)\rightarrow (y,0)}u(x,t)=u_0(y)$ for all $y\in\Omega_0$. Fix $y\in\Omega_0$ and without loss of generality, we may assume that $y=0$ by shifting the origin. For any $x\in\Omega_0$ and any positive $\lambda>0$,
\begin{eqnarray}
\label{3.13}
&&\displaystyle P_{x}\left[\tau_x^{\Omega_t,-v}< t\right]\leq P_{x}\left[\max_{s\in[0,t]}|\theta_t^{s}B_s-x|\geq d(x,\partial\Omega_t)\right]\leq P_{x}\left[\max_{s\in[0,t]}|B_s|\geq d(x,\partial\Omega_t)-|x|-\eta t\right]\nonumber\\
&&=P_{x}\left[\max_{s\in[0,t]}\exp{\{\lambda |B_s|\}}\geq \exp\{\lambda(d(x,\partial\Omega_t)-|x|-\eta t)\}\right].
\end{eqnarray} 
Since $\exp\{\lambda |B_t|\}$ is a positive submartingale, we can apply Doob's inequality, then
\begin{eqnarray}
\label{3.14}
P_{x}\left[\max_{s\in[0,t]}\exp{\{\lambda |B_s|\}}\geq \exp\{\lambda(d(x,\partial\Omega_t)-|x|-\eta t)\}\right]\leq\frac{  E_{x}[\exp{(\lambda |B_t|)}]   }{ \exp\{\lambda(d(x,\partial\Omega_t)-|x|-\eta t)\} }.\nonumber\\
\end{eqnarray} 
By $(\ref{3.13})$ and $(\ref{3.14})$, we obtain
\begin{eqnarray}
\displaystyle\lim_{(x,t)\rightarrow (y,0)}P_{x}\left[\tau_x^{\Omega_t,-v}< t\right] \leq \exp\{-\lambda d(y,\partial\Omega_0)\}
\end{eqnarray} 
so that the left hand side vanishes since $\lambda > 0$ is arbitrary. Thus we deduce that
\begin{eqnarray}
\displaystyle \lim_{(x,t)\rightarrow (y,0)}\int_{\Omega_0}u_0(\xi)G_{0,t}^{\Omega,v}(\xi,x)d\xi=\lim_{(x,t)\rightarrow (y,0)}E_{x}[u_0(B_t)]=u_0(y).
\end{eqnarray} 
Since the Gaussian kernel satisfies the heat equation, we conclude that $u$ is the unique solution of Theorem \ref{thm1}. \\
To show $(\ref{888})$, let us fix $(x,t)\in D_T$. Applying Green's formula, we have
\begin{eqnarray}
\label{3.32}
&&\displaystyle \int_{\Omega_s}u(y,s)\Delta_y G(x,t;y,s)-G(x,t;y,s)\Delta_y u(y,s)dy\nonumber\\
&&\hspace{2cm}=\int_{\partial\Omega_s}u(y,s)\frac{\partial G}{\partial n_{y,s}}(x,t;y,s)-G(x,t;y,s)\frac{\partial u}{\partial n}(y,s)d\mathcal{H}^{d-1}(y).
\end{eqnarray}
Since $u$ is $0$ on the boundary, \eqref{3.32} becomes
\begin{eqnarray}
\label{3.33}
&&\displaystyle \int_{\Omega_s}u(y,s)\Delta_y G(x,t;y,s)-G(x,t;y,s)\Delta_y u(y,s)dy=\int_{\Omega_s}-2(uG)_s dy=-2\frac{\partial}{\partial s}\int_{\Omega_s}uG dy\nonumber\\
&&\hspace{2cm}=\int_{\partial\Omega_s}-G(x,t;y,s)\frac{\partial u}{\partial n}(y,s)d\mathcal{H}^{d-1}(y).
\end{eqnarray}
By integrating the last two terms of \eqref{3.33} with respect to $s$ from $0$ to $t$, we obtain another different representation of $u$ as follows:
\begin{eqnarray}
\label{777}
\displaystyle u(x,t)=\int_{\Omega_0}u_0(\xi)G(x,t;\xi,0)d\xi+\frac{1}{2}\int_{0}^{t}\int_{\partial\Omega_s}G(x,t;y,s)\frac{\partial u}{\partial n}(y,s)d\mathcal{H}^{d-1}(y)ds.
\end{eqnarray}
Taking normal derivatives of both sides in $(\ref{777})$ and applying the jump relation, we get
\begin{eqnarray}
\displaystyle \frac{1}{2}\frac{\partial u}{\partial n}(x,t)=\int_{\Omega_0}u_0(\xi)\frac{\partial G}{\partial n_{x,t}}(x,t;\xi,0)d\xi+\frac{1}{2}\int_{0}^{t}\int_{\partial\Omega_s}\frac{\partial G}{\partial n_{x,t}}(x,t;y,s)\frac{\partial u}{\partial n}(y,s)d\mathcal{H}^{d-1}(y)ds\nonumber\\
\end{eqnarray}
which implies $(\ref{888})$.
\end{proof}

\vskip1cm
\section{Proof of \textbf{of Theorem~\ref{mainthm}}} 
\label{Fsec.4}
\ Comparing the definition \eqref{2345} of $u$ and $(\ref{777})$, using \eqref{12345}, we see the following equality:
\begin{eqnarray}
\displaystyle \int_{[0,t)}\int_{\partial\Omega_s}G_{s,t}(y,x)\int_{\Omega_0}u_0(\xi)dF_{\xi}(y,s)d\xi=-\frac{1}{2}\int_{0}^{t}\int_{\partial\Omega_s}G(x,t;y,s)\frac{\partial u}{\partial n}(y,s)d\mathcal{H}^{d-1}(y)ds.\nonumber\\
\end{eqnarray}
Let us denote $\displaystyle dF_{u_0}(y,s):=\int_{\Omega_0} u_0(\xi) dF_{\xi}(y,s)d\xi$. 
\begin{prop}
\label{new}
$\displaystyle dF_{u_0}(y,s)=-\frac{1}{2}\frac{\partial u}{\partial n}(y,s)d\mathcal{H}^{d-1}(y)ds$.
\end{prop}
\ For the proof of Proposition \ref{new}, we introduce the mass lost $\displaystyle \Delta_{I}^{\Omega,v}(u)$, $I=[t_1,t_2]\subset [0,T]$, $t_1\leq t_2$, is defined by
\begin{eqnarray}
\displaystyle \Delta_{I}^{\Omega,v}(u)=\int_{\Omega_{t_1} }u(x,t_1)dx-\int_{\Omega_{t_2}}u(x,t_2)dx.
\end{eqnarray}
If we see the right hand side of $(\ref{777})$, we can extend $u$ to $\bar{u}$ defined in $\{(x,t): x\in\mathbb{R}^{d},\ 0<t\leq T\}$ as 
\begin{eqnarray}
\displaystyle \bar{u}(x,t)=\int_{\Omega_0}u_0(\xi)G(x,t;\xi,0)d\xi+\frac{1}{2}\int_{0}^{t}\int_{\partial\Omega_s}G(x,t;y,s)\frac{\partial u}{\partial n}(y,s)d\mathcal{H}^{d-1}(y)ds.
\end{eqnarray}
Then this satisfies the heat equation with $\displaystyle \lim_{(x,t)\rightarrow (y,0)}\bar{u}(x,t)=0$ for all $\displaystyle y\in \Omega_0^{c}$ and also satisfies $\bar{u}(x,t)=0$ for all $0<t\leq T$ and all $x\in\partial\Omega_t$. Moreover, by the properties of Gaussian kernel, we have 
\begin{eqnarray}
\displaystyle \lim_{|x|\rightarrow \infty}\sup_{0< t\leq T}|\bar{u}(x,t)|=0.
\end{eqnarray}
It follows that $\bar{u}(x,t)=0$ in $\{(x,t): x\in\Omega_t^{c}, 0< t\leq T\}$ by the weak maximum(minimum) principle. Thus we assume that $u$ is defined $\{(x,t): x\in\mathbb{R}^{d},\ 0\leq t\leq T\}$ such that it is $0$ in $\{(x,t): x\in\Omega_t^{c},\ 0\leq t\leq T\}$.
\\ 
Heuristically
\begin{eqnarray*}
\displaystyle \Delta_{I}^{X}(u)= -\int \int_{t_1}^{t_2} u_t(x,t)dtdx= -\int \int_{t_1}^{t_2}\frac{1}{2}\Delta_x u(x,t)dxdt= -\frac{1}{2}\int_{t_1}^{t_2}\int_{\partial\Omega_t} \frac{\partial u}{\partial n}(x,t)d\mathcal{H}^{d-1}(x)dt.
\end{eqnarray*}
Since we do not control $\Delta_x u$ at the moving boundary, we cannot make this argument rigorously. Thus we use a different approach.\\

\begin{proof}\hspace{-1mm}
\textbf{of Proposition~\ref{new}}
\vskip 0.1cm
It suffices to show
$$\displaystyle -\frac{1}{2}\int_{I}\int_{\partial\Omega_t}\frac{\partial u}{\partial n}(x,t)d\mathcal{H}^{d-1}(x)dt=\Delta_{I}^{\Omega,v}(u)=\int_{I}\int_{\partial\Omega_t}dF_{u_0}(x,t).$$
If we integrate both sides of $(\ref{777})$, then
\begin{eqnarray*}
\displaystyle \int_{\mathbb{R}^{d}}u(x,t)dx= \int_{\mathbb{R}^{d}}\int_{\Omega_0}u_0(\xi)G(x,t;\xi,0)d\xi dx+\frac{1}{2} \int_{\mathbb{R}^{d}}\int_{0}^{t}\int_{\partial\Omega_s}G(x,t;y,s)\frac{\partial u}{\partial n}(y,s)d\mathcal{H}^{d-1}(y)ds.
\end{eqnarray*}
Applying Fubini's theorem, we get
\begin{eqnarray}
\displaystyle   \int_{\Omega_t}u(x,t)dx=\int_{\mathbb{R}^{d}}u(x,t)dx= \int_{\Omega_0}u_0(\xi)d\xi+\frac{1}{2}\int_{0}^{t}\int_{\partial\Omega_s}\frac{\partial u}{\partial n}(y,s)d\mathcal{H}^{d-1}(y)ds.
\end{eqnarray}
Thus we get the first equality of the proposition,
\begin{eqnarray}
\displaystyle \Delta_{[t_1,t_2]}^{X}(u)=-\frac{1}{2}\int_{t_1}^{t_2}\int_{\partial\Omega_s}\frac{\partial u}{\partial n}(y,s)d\mathcal{H}^{d-1}(y)ds.
\end{eqnarray}
By $(\ref{1.0})$ and item $\ref{1.1}$ of Theorem \ref{mainthm}, we get
\begin{eqnarray}
\displaystyle P_{\xi}[\tau_\xi^{X} \geq t]=\int_{\Omega_t}G_{0,t}^{\Omega,v}(\xi,x)dx.
\end{eqnarray}
For $0=t_1<t_2$, using Fubini's theorem again, we get
\begin{eqnarray*}
\displaystyle \Delta_{I}^{\Omega,v}(u)&=&\int_{\Omega_0}u_0(\xi)d\xi-\int_{\Omega_{t_2}}\int_{\Omega_0} u_0(\xi)G_{0,t_2}^{\Omega,v}(\xi,x)d\xi dx\\
&=&\int_{\Omega_0}u_0(\xi)d\xi-\int_{-\infty}^{0}u_0(\xi)P_{\xi}[\tau_\xi^{\Omega,v} \geq t_{2}]d\xi=\int_{-\infty}^{0}u_0(\xi)P_{\xi}[0\leq\tau_\xi^{X} < t_{2}]d\xi.\\
\end{eqnarray*}
For $0<t_1<t_2$, similarly,
\begin{eqnarray*}
\displaystyle \Delta_{I}^{\Omega,v}(u)&=&\int_{\Omega_{t_1}}\int_{\Omega_0} u_0(\xi)G_{0,t_1}^{\Omega,v}(\xi,x)d\xi dx-\int_{\Omega_{t_2}}\int_{\Omega_0} u_0(\xi)G_{0,t_2}^{\Omega,v}(\xi,x)d\xi dx\\
&=&\int_{\Omega_0}u_0(\xi)P_{\xi}[t_1\leq \tau_\xi^{\Omega,v} < t_{2}]d\xi.
\end{eqnarray*}
Then for $I_\epsilon=[t_2, t_2+\epsilon]$, we get
\begin{eqnarray*}
\displaystyle \lim_{\epsilon\rightarrow 0}\Delta_{I_\epsilon}^{\Omega,v}(u)&=&\lim_{\epsilon\rightarrow 0}-\frac{1}{2}\int_{I_\epsilon}\int_{\partial\Omega_t}\frac{\partial u}{\partial n}(x,t)d\mathcal{H}^{d-1}(x)dt=0=\lim_{\epsilon\rightarrow 0}\int_{\Omega_0}u_0(\xi)P_{\xi}[t_2\leq \tau_\xi^{\Omega,v} < t_2+\epsilon]d\xi\\
&=&\int_{\Omega_0}u_0(\xi)P_{\xi}[\tau_\xi^{\Omega,v} = t_{2}]d\xi.
\end{eqnarray*}
Finally we conclude that
\begin{eqnarray}
\displaystyle \Delta_{I}^{\Omega,v}(u)=\int_{\Omega_0}u_0(\xi)P_{\xi}[t_1\leq \tau_\xi^{X}\leq t_2]d\xi=\int_{I}\int_{\partial\Omega_s}dF_{u_0}(y,s)ds.
\end{eqnarray}
\end{proof}



By approximating the initial delta measure of Theorem \ref{mainthm}, we prove the proposition below.

\begin{prop}
\label{prop.5}
Let us assume the starting point of Brownian motion $r_0=0\in\Omega_0$ and let us choose a sequence $\displaystyle\{h_m\}\subset C_c^{\infty}(\mathbb{R}^{d};\mathbb{R}_{+})$ with $\textrm{supp}\hspace{0.1cm}{h_m}=\mathcal{B}(0, \frac{1}{m})\Subset \Omega_0$ and $\displaystyle \lVert h_m \rVert_{1}=1$. For each $h_m$, there exists a corresponding solution $\displaystyle u_{m}$ with $\displaystyle -\frac{1}{2}\frac{\partial u_m}{\partial n}(x,t)=:p_{m}(x,t)$ in the sense of Theorem \ref{thm3}. Then we have the following statements:
\begin{enumerate}
\label{100} 
\item There is a unique $p\in C(S_T)$ with $\displaystyle p(\cdot,0)=0$ such that for all $t>0$ and all $x\in\Omega_t$,
\begin{eqnarray}
\label{111}
\displaystyle p(x,t)=-\frac{\partial G}{\partial n _{x,t}}(x,t;0,0)+\int_0^{t}\int_{\partial\Omega_s}\frac{\partial G}{\partial n_{x,t}}(x,t;y,s)p(y,s)d\mathcal{H}^{d-1}(y)ds.
\end{eqnarray}
\label{222}
\item $p_m$ converges to $p$ in sup norm.
\end{enumerate}
\end{prop}


\begin{proof}\hspace{-1mm}
\textbf{of Proposition~\ref{prop.5}}\\
For all $T$ sufficiently small, by Lemma \ref{lem.1} and \ref{lem.2}, we obtain that there is $C_1>0$ such that for all $0<t\leq T$ and all $x\in\partial\Omega_t$,
\begin{eqnarray}
\displaystyle \left|\int_{0}^{t}\int_{\partial\Omega_\tau}\frac{\partial G}{\partial n_{x,t}}(x,t;y,\tau)d\mathcal{H}^{d-1}(y)d\tau\right|\leq C_1(t-\tau)^{2\gamma-\frac{1}{2}}.
\end{eqnarray}
For $T_s>0$, we define $\displaystyle \mathcal{F}:C(S_{T_s})\rightarrow C(S_{T_s})$ such that for $q\in C(S_{T_s})$,
\begin{eqnarray}
\displaystyle (\mathcal{F}q)(x,t)=-\frac{\partial G}{\partial n _{x,t}}(x,t;0,0)+\int_0^{t}\int_{\partial\Omega_\tau}\frac{\partial G}{\partial n_{x,t}}(x,t;y,\tau)q(y,\tau)d\mathcal{H}^{d-1}(y)d\tau,
\end{eqnarray}
and  $(\mathcal{F}q)(\cdot,0)=0$. If we choose $T_s$ sufficiently small, then $\mathcal{F}$ is a contraction mapping so that $\mathcal{F}$ has a unique fixed point. Let's call this $\displaystyle p_{T_s}$.\\
Now we have $p_{T_s}$ for some $T_s>0$. For $\displaystyle T^{\star}>T_s$, let us denote $\displaystyle S_{[T_s,T^{\star}]}:=\bigcup_{T_s\leq t\leq T^{\star}}\partial\Omega_t\times\{ t \}$. We define $\displaystyle \mathcal{K}:C(S_{[T_s,T^{\star}]})\rightarrow C(S_{[T_s,T^{\star}]})$ as, for $q\in C(S_{[T_s,T^{\star}]})$,
\begin{eqnarray*}
&&\displaystyle (\mathcal{K}q)(x,t)=-\frac{\partial G}{\partial n _{x,t}}(x,t;0,0)+\int_0^{T_s}\int_{\partial\Omega_\tau}\frac{\partial G}{\partial n_{x,t}}(x,t;y,\tau)p_{T_s}(y,\tau)d\mathcal{H}^{d-1}(y)d\tau\\
&&\hspace{2cm}+\int_{T_s}^{t}\int_{\partial\Omega_\tau}\frac{\partial G}{\partial n_{x,t}}(x,t;y,\tau)q(y,\tau)d\mathcal{H}^{d-1}(y)d\tau.
\end{eqnarray*}
Then for $\displaystyle q_1, q_2\in C(S_{[T_s,T^{\star}]})$, we have
\begin{eqnarray}
\displaystyle \lVert \mathcal{K}q_1-\mathcal{K}q_2\rVert_{\infty}\leq C_2(t-T_s)^{2\gamma-\frac{1}{2}}\lVert q_1-q_2\rVert_{\infty}\leq C_2(T^{\star}-T_s)^{2\gamma-\frac{1}{2}}\lVert q_1-q_2\rVert_{\infty}.
\end{eqnarray}
Similarly, if we choose $\displaystyle T^{\star}$ such that $\displaystyle  C_2(T^{\star}-T_s)^{2\gamma-\frac{1}{2}}<1$, then $\mathcal{K}$ is a contraction mapping so that $\mathcal{K}$ has a unique fixed point.\\
Therefore, if we have $p$ defined $S_{T_s}$, $p_{T_s}$,  then we can extend this to time $T_s+C_3$ where $C_3$ is a constant. Thus if we repeat this step inductively, we have $p$ defined on $S_T$ which satisfies $(\ref{111})$.
\\
We now prove that $p_m$ converges to $p$ in sup norm for all sufficiently small $T_s>0$. By $(\ref{888})$, 
\begin{eqnarray}
\label{999}
\displaystyle p_m(x,t)=-\int_{\Omega_0}h_m(\xi)\frac{\partial G}{\partial n_{x,t}}(x,t;\xi,0)d\xi+\int_{0}^{t}\int_{\partial\Omega_\tau}\frac{\partial G}{\partial n_{x,t}}(x,t;y,\tau)p_m(y,\tau)d\mathcal{H}^{d-1}(y)d\tau.\nonumber\\
\end{eqnarray}
For $0<t\leq T_s$, taking the difference between $(\ref{111})$ and $(\ref{999})$, we get
\begin{eqnarray*}
|p_m(x,t)-p(x,t)|&\leq& \left|\int_{\Omega_0}h_m(\xi)\left(\frac{\partial G}{\partial n_{x,t}}(x,t;\xi,0)-\frac{\partial G}{\partial n_{x,t}}(x,t;0,0)\right)d\xi\right|\\
&+&\left|\int_{0}^{t}\int_{\partial\Omega_\tau}\frac{\partial G}{\partial n_{x,t}}(x,t;y,\tau)(p_m(y,\tau)-p(y,\tau))d\mathcal{H}^{d-1}(y)d\tau\right|.
\end{eqnarray*}
Let us denote  $\displaystyle \lVert p_m-p\rVert_{T_s}:=\sup_{\tau\in[0,T_s]}\sup_{y\in\partial\Omega_\tau}|p_m(y,\tau)-p(y,\tau)|$. For the second term of the right hand side, we have
\begin{eqnarray*}
\left|\int_{0}^{t}\int_{\partial\Omega_\tau}\frac{\partial G}{\partial n_{x,t}}(x,t;y,\tau)(p_m(y,\tau)-p(y,\tau))d\mathcal{H}^{d-1}(y)d\tau\right|\leq C_4\int_{0}^{t}\frac{\lVert p_m-p\rVert_{T_s}}{(t-\tau)^{\frac{3}{2}-2\gamma}}d\tau\\
=C_5t^{2\gamma-\frac{1}{2}}\lVert p_m-p\rVert_{T_s}\leq C_5T_s^{2\gamma-\frac{1}{2}}\lVert p_m-p\rVert_{T_s}.
\end{eqnarray*}
Let us choose $T_s>0$ such that $C_5T_s^{2\gamma-\frac{1}{2}}<1$. Then
\begin{eqnarray*}
\displaystyle (1-C_5T_s^{2\gamma-1})\lVert p_m-p\rVert_{T_s}&\leq&\sup_{0<t\leq T_s}\sup_{x\in\partial\Omega_t}\left|\int_{\Omega_0}h_m(\xi)\left(\frac{\partial G}{\partial n_{x,t}}(x,t;\xi,0)-\frac{\partial G}{\partial n_{x,t}}(x,t;0,0)\right)d\xi\right|.\\
\displaystyle &\leq&\sup_{0<t\leq T_s}\sup_{x\in\partial\Omega_t}\sup_{|\xi|\leq\frac{1}{m}}\left|\frac{\partial G}{\partial n_{x,t}}(x,t;\xi,0)-\frac{\partial G}{\partial n_{x,t}}(x,t;0,0)\right|.
\end{eqnarray*}
We have
\begin{eqnarray}
\label{463}
&&\displaystyle|\partial_i^{2} G(x,t;\xi,0)|=\left|\left\{\left(\frac{\xi_i-x_i}{t}\right)^{2}-\frac{1}{t}\right\}  \left(\frac{1}{\sqrt{2\pi t}}\right)^{d}\exp\left\{-\frac{|x-\xi|^{2}}{2t}\right\}\right|\leq\frac{C_6}{|x-\xi|^{d+2}},\nonumber\\
&&\displaystyle|\partial_i^{2} G(x,t;\xi,0)|\leq \frac{C_7}{t^{\frac{d}{2}+1}},\nonumber\\
&&\displaystyle|\partial_j\partial_i G(x,t;\xi,0)|=\left|\frac{(\xi_i-x_i)(\xi_j-x_j)}{t^2} \left(\frac{1}{\sqrt{2\pi t}}\right)^{d}\exp\left\{-\frac{|x-\xi|^{2}}{2t}\right\}\right|\leq\frac{C_8}{|x-\xi|^{d+2}},\nonumber\\
&&\displaystyle|\partial_j\partial_i G(x,t;\xi,0)|\leq \frac{C_9}{t^{\frac{d}{2}+1}}.
\end{eqnarray}

For all sufficiently small $T_s$ and all sufficiently large $m$, by the mean value theorem and \eqref{463}, there exists $C_{10}>0$ such that for all $|\xi|\leq\frac{1}{m}$,
\begin{eqnarray*}
\left|\frac{\partial G}{\partial n_{x,t}}(x,t;\xi,0)-\frac{\partial G}{\partial n_{x,t}}(x,t;0,0)\right|\leq\frac{C_{10}|\xi|}{(|x|-\frac{1}{m})^{d+2}}\leq \frac{C_{10}}{m(|x|-\frac{1}{m})^{d+2}}
\end{eqnarray*}

Therefore, we conclude that $p_m$ converges to $p$ in sup norm for all sufficiently small $T_s>0$.\\
To extend from $T_s$ to $T^{\star}$, assuming that  $p_m$ converges to $p$ in $\displaystyle C_{S_{T_s}}$ for some $T_s>0$ and writing $\displaystyle \lVert p_m-p\rVert_{[T_s,T^{\star}]}=\sup_{\tau\in[T_s,T^{\star}]}\sup_{y\in\partial\Omega_s}|p_m(y,\tau)-p(y,\tau)|$, for $\displaystyle T_s\leq t\leq T^{\star}$, we deduce that by the mean value theorem and \eqref{463},
\begin{eqnarray*}
|p_m(x,t)-p(x,t)|&\leq& \left|\int_{\Omega_0}h_m(\xi)\left(\frac{\partial G}{\partial n_{x,t}}(x,t;\xi,0)-\frac{\partial G}{\partial n_{x,t}}(x,t;0,0)\right)d\xi\right|\\
&+&\left|\int_{0}^{T_s}\int_{\partial\Omega_\tau}\frac{\partial G}{\partial n_{x,t}}(x,t;y,\tau)(p_m(y,\tau)-p(y,\tau))d\mathcal{H}^{d-1}(y)d\tau\right|\\
&+&\left|\int_{T_s}^{t}\int_{\partial\Omega_\tau}\frac{\partial G}{\partial n_{x,t}}(x,t;y,\tau)(p_m(y,\tau)-p(y,\tau))d\mathcal{H}^{d-1}(y)d\tau\right|\\
&\leq& \frac{C_{11}}{mT_s^{\frac{d}{2}+1}}+C_1\int_{0}^{T_s}\frac{\lVert p_m-p\rVert_{T_s}}{(t-\tau)^{\frac{3}{2}-2\gamma}}d\tau+C_1\int_{T_s}^{t}\frac{\lVert p_m-p\rVert_{[T_s,T^{\star}]}}{(t-\tau)^{\frac{3}{2}-2\gamma}}d\tau\\
&\leq& \frac{C_{11}}{mT_s^{\frac{d}{2}+1}}+C_1\int_{0}^{T_s}\frac{\lVert p_m-p\rVert_{T_s}}{(T_s-\tau)^{\frac{3}{2}-2\gamma}}d\tau+C_{12}\lVert p_m-p\rVert_{[T_s,T^{\star}]}(t-T_s)^{2\gamma-\frac{1}{2}}\\
&\leq&\frac{C_{11}}{mT_s^{\frac{d}{2}+1}}+C_{13}T_s^{2\gamma-\frac{1}{2}}\lVert p_m-p\rVert_{T_s}+C_{12}\lVert p_m-p\rVert_{[T_s,T^{\star}]}(T^{\star}-T_s)^{2\gamma-\frac{1}{2}}.
\end{eqnarray*}
Let us choose $T^{\star}>T_s$ such that $C_{12}(T^{\star}-T_s)^{2\gamma-\frac{1}{2}}<1$, then we have
\begin{eqnarray*}
(1-C_{12}(T^{\star}-T_s)^{\gamma-\frac{1}{2}})\lVert p_m-p\rVert_{[T_s,T^{\star}]}\leq\frac{C_{11}}{mT_s^{\frac{d}{2}+1}}+C_{13}T_s^{2\gamma-\frac{1}{2}}\lVert p_m-p\rVert_{T_s}.
\end{eqnarray*}
The right term vanishes when $m$ goes to $\infty$ so that $p_m$ converges to $p$ in $\displaystyle C(S_{T_s+C_{14}})$ for some constant $C_{14}>0$. By repeating this argument inductively, it follows that $p_m$ converges to $p$ in sup norm.
\end{proof}

\vskip 0.5cm
Now we can show that $p$ is the continuous density function of $\displaystyle dF_{r_0}$.\\
By Proposition \ref{prop.5}, we have
\begin{eqnarray*}
\displaystyle \lim_{m\rightarrow \infty}\int_{\Omega_0}h_{m}(\xi)P_{\xi}(\tau_\xi^{\Omega,v}\in I)d\xi=\lim_{m\rightarrow \infty}\int_{I}\int_{\partial\Omega_t}p_m(x,t)dt=\int_{I}\int_{\partial\Omega_t}p(x,t)dt.
\end{eqnarray*}
For $I=[0,t]\subset[0,T]$, therefore, we obtain that
\begin{eqnarray}
\displaystyle \int_{[0,t]}\int_{\partial\Omega_s}dF_{r_0}(y,s)=P_{r_0}(\tau_{r_0}^{\Omega,v}\leq t)=\int_{0}^{t}\int_{\partial\Omega_s}p(y,s)ds,
\end{eqnarray}
which implies that $p$ is the continuous density function of $\displaystyle dF_{r_0}(y,s)$, thus item $\ref{1.2}$ is proved.\\ 
By Theorem \ref{thm1} and the properties of the Gaussian kernel, $\displaystyle G_{0,t}^{\Omega,v}(r_0,x)$ solves $(\ref{1})$, $(\ref{2})$ and $(\ref{3})$. Hence $G^{\Omega,v}$ is the Green function of the heat equation with Dirichlet boundary condition which implies item $\ref{1.4}$. Furthermore, $\displaystyle G_{0,t}^{\Omega,v}(r_0,x)$ can be written as
\begin{eqnarray}
\displaystyle G_{0,t}^{\Omega,v}(r_0,x)=G_{0,t}(r_0,x)-\int_{0}^{t}\int_{\partial\Omega_\tau}G_{\tau,t}(y,x)p(y,\tau)d\mathcal{H}^{d-1}(y)d\tau.
\end{eqnarray}
Taking a normal derivative with respect to x, applying the jump relation and \eqref{111}, we have
\begin{eqnarray}
\displaystyle \frac{\partial}{\partial n}G_{0,t}^{\Omega,v}(r_0,x)&=&\frac{\partial}{\partial n}G_{0,t}(r_0,x)-p(x,t)-\int_{0}^{t}\int_{\partial\Omega_\tau} \frac{\partial}{\partial n}G_{\tau,t}(y,x)p(y,\tau)d\mathcal{H}^{d-1}(y)d\tau\nonumber\\
&=&-2p(x,t).
\end{eqnarray}
Thus
$\displaystyle p(x,t)=-\frac{1}{2}\frac{\partial}{\partial n}G_{0,t}^{\Omega,v}(r_0,x)$, so it proves item $\ref{1.3}$ of Theorem \ref{mainthm}.

\section*{Acknowledgement}

\end{document}